\documentclass{article}
\usepackage{amsfonts,amsthm}
\usepackage{authblk}
\usepackage{newpxtext,newpxmath}
\usepackage{pdflscape}
\usepackage{multirow}
\usepackage{multicol}
\usepackage{tikz}
\usetikzlibrary{arrows}
\usepackage{geometry}
\usepackage{relsize}

\usepackage{arydshln}
\makeatletter
\newcommand{\oset}[3][0ex]{%
  \mathrel{\mathop{#3}\limits^{
    \vbox to#1{\kern-2\ex@
    \hbox{$\scriptstyle#2$}\vss}}}}
\makeatother

\usepackage{mathrsfs}

\theoremstyle{plain}
\newtheorem{theorem}{Theorem}[section]
\newtheorem{lemma}[theorem]{Lemma}
\newtheorem{proposition}[theorem]{Proposition}

\theoremstyle{definition}
\newtheorem{definition}{Definition}[section]
\newtheorem{notation}{Notation}[section]

\theoremstyle{remark}
\newtheorem{remark}{Remark}[section]

\title{On $(\alpha,1,0)$-derivations of anti-commutative algebras}

\makeatletter
\def\blfootnote{\xdef\@thefnmark{}\@footnotetext}
\makeatother

\makeatletter
\renewcommand\@date
{{%

\vspace{-1.5cm}%
\large\centering

  \begin{tabular}{@{}c@{}}
    Edison Alberto \textsc{Fern\'andez-Culma} \textsuperscript{$\dagger$}\blfootnote{${}^{\dagger}$ Partially supported by CONICET and SECyT-UNC: 33620180100523CB} \\
{\small Centro de Investigaci\'on y Estudios en Matem\'atica de C\'ordoba, CONICET}\\
{\small Facultad de Matem\'atica, Astronom\'ia, F\'isica y Computaci\'on, UNC}\\
{\small C\'ordoba, Argentina}\\
    \texttt{\small efernandez@famaf.unc.edu.ar}
  \end{tabular}%
}}
\makeatother

\author{}

\begin{document}

\maketitle


\begin{abstract}
We solve an open problem concerning the well-known $(\alpha,\beta,\gamma)$-derivations, proving that the spaces of $(\alpha,1,0)$-derivations of any Lie algebra are isomorphic ($\alpha\neq 0,1$). Also, we prove sharp bounds for the \textit{invariants functions} defined by such spaces.
\end{abstract}

\section{Introduction}{$ $}

For simplicity, we assume throughout this paper that all $\mathbb{K}$-algebras are finite dimensional and that the base field $\mathbb{K}$ is of characteristic zero.

\begin{definition}
Let $\mathbb{K}$ be a field. An \textit{anti-commutative algebra over} $\mathbb{K}$ is a vector space over $\mathbb{K}$, say $V$, together with a bilinear  product map $\mu:V \times V \rightarrow V$
satisfying the identity
  $$
  \mu(X,Y) = - \mu(Y,X)  \quad \mbox{ for all } X,Y\in V.
  $$
From now on, we indulge in the typical abuse of notation of identifying the algebra $\mathfrak{A}=(V,\mu)$ with its underlying vector space $V$.
\end{definition}

\begin{notation}
Let $U$, $V$ and $W$ be vector spaces over $\mathbb{K}$. Let us write by $L(V;W)$ the set of all linear transformations from $V$ to $W$, and let $\Theta_{V;W} \in L(V;W)$ be denote \textit{the zero map from $V$ to $W$}, which is defined by $\Theta_{V;W}(X)= 0_{W}$ for all $X \in V$.
If $W$ is a vector subspace of $V$, we denote by $\iota_{W;V}$ (or simply $\iota_{W}$ when no confusion can arise) the \textit{inclusion map} from $W$ to $V$,
and if $U$ and $W$ are \textit{complementary subspaces} of $V$; i.e. $V$ is an \textit{internal direct sum} of $U$ and $W$, we write $\pi_{U}^{U \oplus W}$ and $\pi_{W}^{U \oplus W}$ for the \textit{projection maps} from $V$
onto $U$ and $W$, respectively (or simply $\pi_{U}$ and $\pi_{W}$ when the decomposition $V=U \oplus W$ is clear from the context).
\end{notation}

\begin{definition}[{\cite[\S 2]{NovotnyHrivnak}}]
Let $\mathfrak{A}=(V,\mu)$ be an anti-commutative $\mathbb{K}$-algebra and let $\alpha, \beta, \gamma \in \mathbb{K}$ be three constants. An  \textit{$(\alpha,\beta,\gamma)$-derivation} of  $\mathfrak{A}$ is a linear transformation $D$ from $V$ to itself, such that
$$
\alpha D\mu(X,Y) = \beta \mu (DX,Y) + \gamma \mu (X,DY)
$$
for all $X, Y \in V$.
The set all $(\alpha,\beta,\gamma)$-derivations of $\mathfrak{A}$ is denoted by $\mathcal{D}(\alpha,\beta,\gamma)(\mathfrak{A})$.
\end{definition}

If two algebras $\mathfrak{A}$ and $\mathfrak{B}$ are isomorphic, then $\mathcal{D}(\alpha,\beta,\gamma)(\mathfrak{A})$ and $\mathcal{D}(\alpha,\beta,\gamma)(\mathfrak{B})$ are isomorphic vector spaces, and so, $\operatorname{Dim} \mathcal{D}(\alpha,\beta,\gamma)(\mathfrak{A})$ is an \textit{invariant} of algebras. To be more specific, let $C^{2}(V;V)$ denote the vector space consisting of all antisymmetric bilinear product maps of $V$ and let $\operatorname{GL}(V)$ be the \textit{General Linear Group} of the vector space $V$. Consider the natural action of $\operatorname{GL}(V)$ on $C^{2}(V;V)$ by \textit{change of basis}:
$$
g\cdot \mu( X, Y):= g \mu (g^{-1}X, g^{-1} Y).
$$
It follows that two algebras $\mathfrak{A}=(V,\mu)$ and $\mathfrak{B}=(V,\lambda)$ are isomorphic if and only if $\mu$ and $\lambda$ are in the same $\operatorname{GL}(V)$-orbit, and that the function $\chi: C^{2}(V;V) \rightarrow [\!|0,  n^2|\!]:= \{0,1, \ldots, n^2\}$ defined by $\chi(\mu) = \operatorname{Dim}\mathcal{D}(\alpha,\beta,\gamma)(V,\mu)$ is an \textit{invariant function} for the action (where $n$ denotes the dimension of $V$). By abuse of notation, if $\tau$ is an invariant function, we write sometimes $\tau(\mathfrak{A})$ instead of $\tau(\mu)$.

The $(\alpha,\beta,\gamma)$-derivations are a particular case of \textit{extended derivations} (see \cite{FernandezCulma}), which can be used to study \textit{degenerations of algebras} (a very active research topic with applications in mathematical physics, differential geometry, harmonic analysis and algebra).

It is an open problem (posed in \cite[\S. 4]{NovotnyHrivnak}, fifth remark) to determine the behavior of the one-parameter family of invariant functions $\{\phi_{n,t}\}_{t \in \mathbb{K}}$, $\phi_{n,t}: C^{2}(V;V) \rightarrow [\!|0,  n^2|\!]$ given by $\phi_{n,t}(\mu)=\operatorname{Dim}\mathcal{D}(t,1,0)(V,\mu)$. Is it an infinite family of functions? If $\widehat{\phi}_{n,t}$ is the \textit{restriction} of $\phi_{n,t}$ to $\mathcal{L}(V)$, \textit{the variety of Lie algebra laws on $V$}, is $\{\widehat{\phi}_{n,t}\}_{t \in \mathbb{K}}$ a finite set? The aim of this paper is to answer both questions and give sharp bounds of $\widehat{\phi}_{n,t}$.

\section*{Acknowledgements}
The author is greatly indebted to Nadina Rojas for many stimulating conversations and for her invaluable support.

\section{The family $\{\phi_{n,t}\}_{t \in \mathbb{K}}$ is infinite }$ $

 To answer the first one, we will give a one-parameter family of pairwise non-isomorphic anti-commutative algebras of dimension $n$ which can be distinguished by using the family $\{\phi_{n,t}\}_{t \in \mathbb{K}}$.
 If $t$ is not a root of unity and $t\neq 0,1$, and if $\mathfrak{A}$ is a centerless anti-commutative algebra over an algebraically closed field $\mathbb{K}$, then it is a simple matter to see that any $(t,1,0)$-derivation of $\mathfrak{A}$ is a nilpotent linear transformation. So, we start by assuming that the dimension of $V$ is four and let $\{e_1,e_2,e_3,e_4\}$ be a basis of $V$. Consider the linear transformation $D:V\rightarrow V$ defined by $e_1 \mapsto e_2$, $e_3 \mapsto e_4$ and $e_2, e_4 \mapsto 0$. Now, we can ask which are all four-dimensional anti-commutative algebras such that $D$ is a $(t,1,0)$-derivation; such algebras are of the form:
$$
\left\{
\begin{array}{l}
\mu({e_1}, {e_3}) = a{e_1}+c{e_2}+b{e_3}+d{e_4},\\
\mu({e_1}, {e_4}) =ta{e_2}+tb{e_4}, \,
\mu({e_2}, {e_3}) =ta{e_2}+tb{e_4}.
\end{array}
\right.
$$
Note the dependence of the \textit{structure constants} of such algebras on the parameter $t$.

\begin{proposition}\label{dim4}
Let $\mathfrak{A}_{s}=(V,\mu_{s})$ be the one-parameter family of four-dimensional anti-commutative algebras given by
$$
\mu_{s} =
\left\{
\begin{array}{l}
\mu_{s}({e_1}, {e_3}) = {e_1}, \, \mu_{s}({e_1},{e_4})=s{e_2}, \, \mu_{s}({e_2},{e_3})=s{e_2},
\end{array}
\right.
$$
with $s \in \mathbb{K}$. If $t\neq 0,1$, then the value of the invariant function $\phi_{4,t}$ at $\mathfrak{A}_{s}$ is
$$
\phi_{4,t}(\mathfrak{A}_{s}) =
\left\{
\begin{array}{l}
6, \mbox{ if } $s=0$\\
1, \mbox{ if } $s=t$\\
0, \mbox{ otherwise }
\end{array}
\right.
$$
and therefore $\phi_{4,t_1} = \phi_{4,t_2}$ if and only if $t_1 = t_2$ (with $t_1, t_2 \neq 0,1$).
\end{proposition}

Before proceeding with the proof, let us recall the following definitions:

\begin{definition}[Centralizer and Derived algebra]
  Let $\mathfrak{A}=(V,\mu)$ be an anti-commutative algebra and let $S \subseteq \mathfrak{A}$ be any subset. The \textit{centralizer} in $\mathfrak{A}$ of $S$ is defined as the set:
  $$
  \mathcal{C}_{\mathfrak{A}}(S):= \{X \in \mathfrak{A} : \mu(X,Y)=0, \mbox{ for all } Y\in S\}.
  $$
  The \textit{derived algebra} of $\mathfrak{A}$, denoted by $\mathfrak{A}^{(2)}$, is the \textit{linear span} of the set of all products of pairs of elements of $\mathfrak{A}$.
\end{definition}

\begin{remark}
Centralizers are vector subspaces of $\mathfrak{A}$ and they need not be subalgebras on account of the lack of associativity or similar properties (such as \textit{Jacobi identity}). If $t\neq0$ and $D$ is a $(t,1,0)$-derivation of $\mathfrak{A}$, then it is clear that $D$ preserves centralizers; and in particular \textit{the center} of $\mathfrak{A}$, $\mathcal{Z}(\mathfrak{A}):= \mathcal{C}_{\mathfrak{A}}(\mathfrak{A})$.
\end{remark}

\begin{remark}
It is worth remarking that if $D$ is a $(t,0,1)$-derivation of $\mathfrak{A}=(V,\mu)$, then $tD\mu(X,Y)=\mu(X,DY)$ for all $X,Y \in V$ (because of the anti-commutative property). Therefore, $\mathcal{D}(t,1,0)(\mathfrak{A}) \subseteq \mathcal{D}(0,1,-1)(\mathfrak{A})$.
\end{remark}

\begin{proof}
Let $A:V \rightarrow V$ be a linear transformation defined by $A e_{j} = \sum_{i=1}^{4} A_j^i e_i$, for $j=1,\ldots,4$, where $A_j^i$ are constants in $\mathbb{K}$. We want to determine the conditions that $A$ must satisfy to be in $\mathcal{D}(t,1,0)(\mathfrak{A}_{s})$. We consider the case $s\neq 0$, in which the derived algebra of $\mathfrak{A}_{s}$ is generated by $\{e_1, e_2\}$ and the centralizer of $e_3$ is generated by $\{e_3, e_4\}$. Suppose that $A \in \mathcal{D}(t,1,0)(\mathfrak{A}_{s})$. Since $t\neq 0$, any $(t,1,0)$-derivation preserves the derived algebra and centralizers, we obtain $A_j^i=0$ for each $i=3,4$, $j=1,2$, and also for each $i=1,2$, $j=3,4$. This implies that $A_2^1=0$ and $A_4^3=0$, because $\mu_{s}(A e_2,e_4) = \mu_{s}(e_2,A e_4) = t D\mu_{s}(e_2,e_4)=0$; and so $A e_2 \in \operatorname{span}\{e_1,e_2\} \cap \mathcal{C}_{\mathfrak{A}_{s}}(e_4)$ and $A e_4 \in \operatorname{span}\{e_3,e_4\} \cap \mathcal{C}_{\mathfrak{A}_{s}}(e_2)$.

Next we study the equality between $t A \mu_{s}(e_1,e_3)= t A e_1 = tA_1^1e_1 + tA_1^2 e_2 $, $\mu_{s}(A e_1, e3)= A_1^1e_1 + sA_1^2 e_2$ and $\mu_{s}( e_1, Ae_3)= A_3^3e_1 + s A_3^4 e_2$. Since $t \neq 1$, we have $A_1^1=0$, and so $A_3^3=0$ and also $(t-s)A_1^2=0$, $A_1^2=A_3^4$ (recall $s\neq 0$). It follows immediately that
$A_2^2=0$ and $A_4^4=0$  since $t A \mu_{s}(e_1,e_4) = t A(s e_2) = t s A_2^2 e_2$, $\mu_{s}(Ae_1,e_4)= \mu_{s}(A_1^2 e_2, e_4)=0$ and $\mu_{s}(e_1,Ae_4)= s A_{4}^4 e_2$ must be equal.

Therefore, if $t=s$, then $A_1^2$ is a free variable and $\mathcal{D}(t,1,0)(\mathfrak{A}_{t})$ is a vector space generated by the nilpotent linear transformation $D$ we just defined above, and if $t\neq s$, then $A=0$ and so $\operatorname{dim} \mathcal{D}(t,1,0)(\mathfrak{A}_{s}) = 0$.

The case $s=0$, which corresponds to the Lie algebra $\mathfrak{aff}(\mathbb{K})\times \mathbb{K}^{2}$ (here $\mathfrak{aff}(\mathbb{K}^n)$ is the \textit{{affine Lie algebra}} of $\mathbb{K}^{n}$), is straightforward or it can be solved by using the result of the following Lemma.
\end{proof}

\begin{notation}
 Let $\mathfrak{A}=(V,\mu)$ be an anti-commutative $\mathbb{K}$-algebra. Let us denote by $\Omega(\mathfrak{A})$ the set
 $$
 \{T \in L(V; V) : \operatorname{Im}(T)\subseteq \mathcal{Z}(\mathfrak{A}) \mbox { and } \mathfrak{A}^{(2)}\subseteq \operatorname{Ker}(T)\},
 $$
  which is the intersection of $\mathcal{D}(0,1,0)(\mathfrak{A})$ and $\mathcal{D}(1,0,0)(\mathfrak{A})$, and is isomorphic to $L(\mathfrak{A}/\mathfrak{A}^{(2)}; \mathcal{Z}(\mathfrak{A}))$. Moreover, note that $\Omega(\mathfrak{A})$ is contained in $\mathcal{D}(t,1,0)(\mathfrak{A})$, for any $t \in \mathbb{K}$. 
  The importance of the set $\Omega(\mathfrak{A})$ is in the elementary observation that any $(t,1,0)$-derivation is determined by its action on $\mathfrak{A}^{(2)}$, modulo $\Omega(\mathfrak{A})$, i.e., if $D_1$ and $D_2$ are two  $(t,1,0)$-derivations of $\mathfrak{A}$ that coincide on  $\mathfrak{A}^{(2)}$, then $D_1 - D_2 \in \Omega(\mathfrak{A})$.
\end{notation}

\begin{lemma}\label{lematecnico}
  Let $\mathfrak{A}$ be an anti-commutative $\mathbb{K}$-algebra and suppose $\mathfrak{B}$ is a subalgebra of $\mathfrak{A}$ such that $\mathfrak{A}$ is the direct sum of $\mathfrak{B}$ and $\mathcal{Z}(\mathfrak{A})$. Then
  $\mathcal{D}(t,1,0)(\mathfrak{A})$ is a vector space isomorphic to $\mathcal{D}(t,1,0)(\mathfrak{B}) \times L(\mathfrak{A}/\mathfrak{A}^{(2)}; \mathcal{Z}(\mathfrak{A}))$ if $t\neq 0$.
\end{lemma}

\begin{proof}
Define a linear transformation $\varphi: \mathcal{D}(t,1,0)(\mathfrak{A}) \rightarrow \mathcal{D}(t,1,0)(\mathfrak{B}) $ given by $D \mapsto \pi_{\mathfrak{B}} \circ D \circ \iota_{\mathfrak{B}}$.
To see that $\varphi$ is well defined; i.e. $\pi_{\mathfrak{B}} \circ D \circ \iota_{\mathfrak{B}}$ is a $(t,1,0)$-derivation of $\mathfrak{B}$, let $X, Y \in \mathfrak{B}$ be arbitrary:
 \begin{align*}
  \pi_{\mathfrak{B}} \circ D \circ \iota_{\mathfrak{B}} (t \mu(X,Y) )&=    \pi_{\mathfrak{B}} \circ D (t \mu(X,Y)) && (\text{since $\mathfrak{B}$ subalgebra of $\mathfrak{A}$})\\
                                                                     &=  \pi_{\mathfrak{B}}(\mu(DX,Y)) && (\text{since ${D} \in \mathcal{D}(t,1,0)(\mathfrak{g}) $ })\\
                                                                     &=  \mu(DX,Y) && (\text{since $\mathfrak{A}^{(2)}=\mathfrak{B}^{(2)}$})\\
                                                                     &=  \mu( D \circ  \iota_{\mathfrak{B}} (X) , Y) && (\text{since $X \in \mathfrak{B}$})\\
                                                                     &=  \mu(\pi_{\mathfrak{B}} \circ D \circ  \iota_{\mathfrak{B}} (X)  + \pi_{\mathcal{Z}(\mathfrak{A}) } \circ D \circ  \iota_{\mathfrak{B}} (X),Y) && (\text{since }V=\mathfrak{B} \oplus \mathcal{Z}(\mathfrak{A})) \\
                                                                     &= \mu(\pi_{\mathfrak{B}} \circ D \circ  \iota_{\mathfrak{B}} (X) ,Y) && (\pi_{\mathcal{Z}(\mathfrak{A}) } \circ D \circ  \iota_{\mathfrak{B}} (X) \in \mathcal{Z}(\mathfrak{A}))
 \end{align*}
Now, we note that $\varphi$ passes to the quotient $\mathcal{D}(t,1,0)(\mathfrak{A}) / \Omega(\mathfrak{A})$, yielding a linear transformation $\widetilde{\varphi}: \mathcal{D}(t,1,0)(\mathfrak{A}) / \Omega(\mathfrak{A}) \rightarrow \mathcal{D}(t,1,0)(\mathfrak{B})$.  To see this, let $D_1$ and $D_2$ in $\mathcal{D}(t,1,0)(\mathfrak{A})$ such that $D_1 - D_2 \in \Omega(\mathfrak{A})$, we need to show that $\pi_{\mathfrak{B}} \circ D_1 \circ \iota_{\mathfrak{B}}$ is equal to $\pi_{\mathfrak{B}} \circ D_2 \circ \iota_{\mathfrak{B}} $.  For this purpose, let $X \in \mathfrak{B}$. We have $\pi_{\mathfrak{B}} \circ (D_1 - D_2 )\circ \iota_{\mathfrak{B}} (X)$
is equal to $\pi_{\mathfrak{B}} \circ (D_1 - D_2 )(X)$ and since $\operatorname{Im}(D_1-D2) \subseteq \mathcal{Z}(\mathfrak{A})$, it follows $\pi_{\mathfrak{B}} \circ (D_1 - D_2 )\circ \iota_{\mathfrak{B}} (X) = 0$.

Next we show that $\widetilde{\varphi}$ is actually an isomorphism of vector spaces. To prove that it is injective, suppose $D \in \mathcal{D}(t,1,0)(\mathfrak{A})$ is such that $ \pi_{\mathfrak{B}} \circ D \circ \iota_{\mathfrak{B}} = 0$ and let us prove $D \in \Omega(\mathfrak{A})$. We show first that $DX \in \mathcal{Z}(\mathfrak{A})$ for any $X \in \mathfrak{A}$. Let $Y \in \mathfrak{A}$ be arbitrary,
 \begin{align*}
   \mu(DX,Y) &= tD\mu(X,Y) \\
             &= t\pi_{\mathfrak{B}}\circ D (\mu(X,Y)) && (\text{since $D$ preserves $\mathfrak{A}^{(2)}$, and $\mathfrak{A}^{(2)}=\mathfrak{B}^{(2)}$})\\
             &= t\pi_{\mathfrak{B}}\circ D \circ \iota_{\mathfrak{B}} (\mu(X,Y)) && (\text{again, since $\mathfrak{A}^{(2)}=\mathfrak{B}^{(2)}$})\\
             &= 0.
 \end{align*}
It remains to check $\mathfrak{A}^{(2)} \subseteq \operatorname{Ker}(D)$, which also follows from the preceding calculation, since $t\neq0$.

To show surjectivity, let $\widehat{D}$ be a $(t,1,0)$-derivation of $\mathfrak{B}$ and consider $D= \iota_{\mathfrak{B}}\circ \widehat{D} \oplus \Theta_{\mathcal{Z}(\mathfrak{A});\mathfrak{A}}$ the unique linear transformation from $\mathfrak{A}$ to $\mathfrak{A}$ such that $D \circ \iota_{\mathfrak{B}} = \iota_{\mathfrak{B}}\circ \widehat{D}$ and $D \circ \iota_{\mathcal{Z}(\mathfrak{A})} = \Theta_{\mathcal{Z}(\mathfrak{A});\mathfrak{A}}$ (\textit{Characteristic property of the direct sum}). We want to check that $D$ is a $(t,1,0)$-derivation of $\mathfrak{A}$. Let $X, Y \in \mathfrak{A}$ be arbitrary:
 \begin{align*}
   \mu(DX,Y) &= \mu(D\circ \iota_{\mathfrak{B}} \circ \pi_{\mathfrak{B}}( X),Y) + \mu(D\circ \iota_{\mathcal{Z}(\mathfrak{A})} \circ \pi_{\mathcal{Z}(\mathfrak{A})}( X),Y)
             &&  \\
             &= \mu( \iota_{\mathfrak{B}} \circ \widehat{D} \circ \pi_{\mathfrak{B}}( X), Y) + \mu(\Theta_{\mathcal{Z}(\mathfrak{A});\mathfrak{A}} \circ \pi_{\mathcal{Z}(\mathfrak{A})}( X),Y)
             && (\text{by definition of $D$}) \\
             &=  \mu( \iota_{\mathfrak{B}} \circ \widehat{D} \circ \pi_{\mathfrak{B}}( X), Y) && \\
             & = \mu( \iota_{\mathfrak{B}} \circ  \widehat{D} \circ \pi_{\mathfrak{B}}( X), \pi_{\mathfrak{B}}(Y) ) + \mu( \iota_{\mathfrak{B}} \circ  \widehat{D} \circ \pi_{\mathfrak{B}}( X), \pi_{\mathcal{Z}(\mathfrak{A})}(Y) ) &&  (\text{since } V=\mathfrak{B} \oplus \mathcal{Z}(\mathfrak{A})) \\
             &= \mu(  \widehat{D} \circ \pi_{\mathfrak{B}}( X), \pi_{\mathfrak{B}}(Y) ) && (\text{since $\widehat{D} \circ \pi_{\mathfrak{B}}( X) \in \mathfrak{B}$})\\
             & = t\widehat{D}\mu(\pi_{\mathfrak{B}}( X) , \pi_{\mathfrak{B}}(Y))
             && (\text{since $\widehat{D} \in \mathcal{D}(t,1,0)(\mathfrak{B}) $}  )\\
             & = t\iota_{\mathfrak{B}} \circ \widehat{D} (\mu(\pi_{\mathfrak{B}}( X) , \pi_{\mathfrak{B}}(Y)))
             && (\text{since $\operatorname{Im}\widehat{D} \subseteq \mathfrak{B}$})\\
             & = t D \circ \iota_{\mathfrak{B}}(\mu(\pi_{\mathfrak{B}}( X) , \pi_{\mathfrak{B}}(Y)))
             && (\text{by definition of $D$})\\
             & = t D \mu(\pi_{\mathfrak{B}}( X) , \pi_{\mathfrak{B}}(Y)) && (\text{since $\mathfrak{B}$ is subalgebra of $\mathfrak{A}$})\\
             & = t D \mu(\pi_{\mathfrak{B}}( X) + \pi_{ \mathcal{Z}(\mathfrak{A}) }( X) , \pi_{\mathfrak{B}}(Y) +\pi_{ \mathcal{Z}(\mathfrak{A})}( Y) )&& \\
             & = t D \mu(X,Y).
 \end{align*}
Since $\varphi(D) = \pi_{\mathfrak{B}} \circ D \circ \iota_{\mathfrak{B}} = \pi_{\mathfrak{B}} \circ \iota_{\mathfrak{B}} \circ \widehat{D} = \widehat{D}$, we have $\widetilde{\varphi}([D]) = \widehat{D}$; where $[D]$ stands for the equivalence class of $D$ in $\mathcal{D}(t,1,0)(\mathfrak{A}) / \Omega(\mathfrak{A})$.
\end{proof}

Although the Lemma \ref{lematecnico} can be improved, it is quite sufficient to prove one of our main results.

\begin{theorem}
  For all $n \in \mathbb{N}$ with $n \geq 4$, the familiy $\{\phi_{n,t}\}_{t \in \mathbb{K}}$ is infinite
\end{theorem}

\begin{proof}
Let $m \in \mathbb{N}$ and let $n=4+m$. Consider the one-parameter family of algebras $\mathfrak{A}_{s} \times \mathbb{K}^{m}$, with $s \in \mathbb{K}\setminus\{0\}$, where $\mathfrak{A}_{s}$ is the four-dimensional $\mathbb{K}$-algebra defined in  Propoposition \ref{dim4}. Let $t \in \mathbb{K}$, $t \neq 0,1$. By the preceding lemma, $\mathcal{D}(t,1,0)(\mathfrak{A}_{s} \times \mathbb{K}^{m})$ is isomorphic to $\mathcal{D}(t,1,0)(\mathfrak{A}_{s}) \times L( \operatorname{span}\{e_3,e_4\}\oplus \mathbb{K}^{m};  \mathbb{K}^{m} )$ (since $s\neq 0$). And so, $\phi_{n,t}(\mathfrak{A}_{t} \times \mathbb{K}^{m})= 1 + (m+2)m$ and if $s \neq t$, $\phi_{n,t}(\mathfrak{A}_{s} \times \mathbb{K}^{m})= 0 + (m+2)m$, which implies $\phi_{n,t} \neq \phi_{n,u}$ if $t \neq u$ (with $t,u \in \mathbb{K} \setminus \{0,1\}$).
\end{proof}

\begin{remark}
It is a simple matter to see that the preceding theorem does not hold with $n=3$. Let $\mathfrak{A}$ be an anti-commutative algebra of dimension $3$ over $\mathbb{K}$ and let $t \in \mathbb{K}\setminus\{0,1\}$. If $\mathfrak{A}$ is centerless, then one can see that $\mathcal{D}(t,1,0)(\mathfrak{A}) = \{0\}$. If $\mathcal{Z}(\mathfrak{A})\neq\{0\}$, then $\mathfrak{A}$ is a Lie algebra which is isomorphic to either $\mathbb{K}^3$, $\mathfrak{aff}(\mathbb{K})\times \mathbb{K}$ or the Heisenberg Lie algebra $\mathfrak{h}_{3}(\mathbb{K})$, and
an easy computation shows that $\phi_{3,t}(\mathfrak{h}_{3}(\mathbb{K}))=3$ and $\phi_{3,t}(\mathfrak{aff}(\mathbb{K})\times \mathbb{K})=2$.
\end{remark}

\section{The family $\{\widehat{\phi}_{n,t}\}_{t \in \mathbb{K}}$ is finite and has three elements}$ $

\begin{definition}
  An anti-commutative $\mathbb{K}$-algebra $\mathfrak{g}=(V,\mu)$ is called \textit{Lie algebra} if it satisfies the \textit{Jacobi Identity}: for all $X,Y,Z \in V$
  $$
  \mu(X, \mu(Y,Z)) + \mu(Y, \mu(Z,X)) + \mu(Z, \mu(X,Y))=0.
  $$
  We denote by $\mathcal{L}(V)$ to the \textit{algebraic subset} of $C^{2}(V;V)$ defined by
  $$
  \mathcal{L}(V)=\{\lambda \in C^{2}(V;V) : (V,\lambda) \mbox{ is a Lie algebra}\},
  $$
  which is usually called the \textit{variety of Lie algebras on $V$}.
\end{definition}

We start by proving an interesting algebraic identity which is satisfied by $(t,1,0)$-derivations of Lie algebras, with $t\neq 0,1$.

\begin{lemma}\label{lematecnico2}
 Let $t \in \mathbb{K} \setminus\{0,1\}$ and let $\mathfrak{g}=(V,\mu)$ be a Lie algebra. If $D$ is a $(t,1,0)$-derivation of $\mathfrak{g}$ then for all $X,Y,Z \in \mathfrak{g}$
 $$
 D\mu(X,\mu(Y,Z))= 0
 $$
\end{lemma}

\begin{proof}
Let $X,Y,Z \in \mathfrak{g}$ be artbitrary
 \begin{align*}
 D\mu(Z,\mu(X,Y)) &= \frac{1}{t}\mu(DZ,\mu(X,Y))&&  (\text{since ${D} \in \mathcal{D}(t,1,0)(\mathfrak{g}) $}  ) \\
                  &= \frac{1}{t}\mu(Z,D\mu(X,Y))&&  (\text{since ${D} \in \mathcal{D}(0,1,-1)(\mathfrak{g}) $}  )\\
                  &= \frac{1}{t}\mu(Z,\frac{1}{t}\mu(DX,Y))&&  (\text{since ${D} \in \mathcal{D}(t,1,0)(\mathfrak{g}) $}  )\\
                  &= -\frac{1}{t^2}\mu(DX,\mu(Y,Z))-\frac{1}{t^2}\mu(Y,\mu(Z,DX)) && (\text{by Jacobi Identity}) \\
                  &= -\frac{t}{t^2}D\mu(X,\mu(Y,Z))-\frac{t^2}{t^2}D\mu(Y,\mu(Z,X)) && (\text{since ${D} \in \mathcal{D}(t,1,0)(\mathfrak{g}) \subseteq \mathcal{D}(0,1,-1)(\mathfrak{g}) $}  ).
\end{align*}

The last equality implies $-\frac{1}{t}D\mu(X,\mu(Y,Z))$ is equal to $D\mu(Z,\mu(X,Y))+D\mu(Y,\mu(Z,X))= -D\mu(X,\mu(Y,Z))$ (by linearity and Jacobi Identity again), and since $t\neq 1$, it follows that $D\mu(X,\mu(Y,Z))$ must be $0$.

\end{proof}
The next theorem follows immediately from the preceding lemma.

\begin{theorem}
Let $t\in \mathbb{K}\setminus\{0,1\}$ and let $\mathfrak{g}$ be a \textit{perfect} Lie algebra over $\mathbb{K}$; i.e. $\mathfrak{g}^{(2)}=\mathfrak{g}$. Then $\mathcal{D}(t,1,0)(\mathfrak{g})=\{0\}$.
\end{theorem}

\begin{remark}
  Note that we can obtain other interesting observations from Lemma \ref{lematecnico2}, for instance, any $(t,1,0)$-derivation of a Lie algebra $\mathfrak{g}$ sends $\mathfrak{g}^{(2)}$ to $\mathfrak{g}^{(2)} \cap \mathcal{Z}(\mathfrak{g})$ (if $t\neq 0,1$).
\end{remark}

Now, we focus our attention on non-perfect Lie algebras.

\begin{theorem}
  Let $t,s\in \mathbb{K}\setminus\{0,1\}$ and let $\mathfrak{g}$ be a non-perfect Lie algebra over $\mathbb{K}$. Then $\mathcal{D}(t,1,0)(\mathfrak{g})$ and $\mathcal{D}(s,1,0)(\mathfrak{g})$ are
  isomorphic vector spaces.
\end{theorem}

\begin{proof}
  Since $\mathfrak{g}$ is non-perfect, let $\mathfrak{a}$ be a complementary subspace of $\mathfrak{g}^{(2)}$ in $\mathfrak{g}$; i.e. $\mathfrak{g} = \mathfrak{a} \oplus \mathfrak{g}^{(2)}$. Given $D \in \mathcal{D}(t,1,0)(\mathfrak{g})$, consider the linear transformations  $\frac{s}{t}D\circ \iota_{\mathfrak{g}^{(2)}}$ and $D\circ \iota_{\mathfrak{a}}$, and let $\widehat{D}=\frac{s}{t}D\circ \iota_{\mathfrak{a}} \oplus D\circ \iota_{\mathfrak{g}^{(2)}}$ be the unique linear transformation
  in $L(\mathfrak{g},\mathfrak{g})$ such that $\widehat{D}\circ \iota_{\mathfrak{a}} = \frac{s}{t}D\circ \iota_{\mathfrak{a}}$ and $\widehat{D}\circ \iota_{\mathfrak{g}^{(2)}} =D\circ \iota_{\mathfrak{g}^{(2)}}$.
  We will show that $\widehat{D}$ is a $(s,1,0)$-derivation of $\mathfrak{g}$. Let $X,Y \in \mathfrak{g}$ and suppose first that $X \in \mathfrak{a}$:
 \begin{align*}
   \mu(\widehat{D}X,Y) &= \mu(\widehat{D}\circ \iota_{\mathfrak{a}}X,Y) && \\
                       &= \mu\left(\frac{s}{t} D \circ  \iota_{\mathfrak{a}}X,Y\right) && (\text{by definition of $\widehat{D}$})\\
                       &=\frac{s}{t}tD  \mu(X,Y) && (\text{since $X \in \mathfrak{a}$ and $D\in \mathcal{D}(t,1,0)(\mathfrak{g})$ })\\
                       &= s D \circ \iota_{\mathfrak{g}^{(2)}}\mu(X,Y) && \\
                       &= s \widehat{D}\circ \iota_{\mathfrak{g}^{(2)}} \mu(X,Y) && (\text{by definition of $\widehat{D}$})\\
                       &= s \widehat{D} \mu(X,Y).
 \end{align*}
And if $X \in \mathfrak{g}^{(2)}$, we have
 \begin{align*}
    \mu(\widehat{D}X,Y) &= \mu(\widehat{D}\circ \iota_{\mathfrak{g}^{(2)}}X,Y) && \\
                        &= \mu(D\circ \iota_{\mathfrak{g}^{(2)}}X, Y)   && (\text{by definition of $\widehat{D}$})\\
                        &= \mu(DX,Y) && \\
                        &= tD\mu(X,Y) && \\
                        &= 0          && (\text{by Lemma \ref{lematecnico2}, since $X\in \mathfrak{g}^{(2)}$})
 \end{align*}
and
\begin{align*}
  s\widehat{D}\mu(X,Y) &= s \widehat{D} \circ \iota_{\mathfrak{g}^{(2)}}\mu(X,Y) &&\\
                       &= s D \circ \iota_{\mathfrak{g}^{(2)}}\mu(X,Y) && (\text{by definition of $\widehat{D}$})\\
                       &= s D\mu(X,Y) && \\
                       &= 0 && (\text{by Lemma \ref{lematecnico2}, since $X\in \mathfrak{g}^{(2)}$}).
\end{align*}
Thus, $s\widehat{D}\mu(X,Y) = \mu(\widehat{D}X,Y) $ when $X \in  \mathfrak{g}^{(2)}$. It follows by linearity that $\widehat{D}$ is a $(s,1,0)$-derivation of $\mathfrak{g}$. 

And so, the function $\varphi:\mathcal{D}(t,1,0)(\mathfrak{g}) \rightarrow \mathcal{D}(s,1,0)(\mathfrak{g})$ that sends a $D$ to $\widehat{D}$ is well defined and is
an isomorphism of vector spaces. Note that we have not already used the hypothesis that $s\neq 1$, which proves the Theorem \ref{centroid} below.
 
When $s\neq 0, 1$, the function  $\psi:\mathcal{D}(s,1,0)(\mathfrak{g}) \rightarrow \mathcal{D}(t,1,0)(\mathfrak{g})$ given by $\psi(\widehat{D})= \frac{t}{s}\widehat{D}\circ \iota_{\mathfrak{a}} \oplus \widehat{D}\circ \iota_{\mathfrak{g}^{(2)}}$ is the inverse of $\varphi$, which completes the proof.
\end{proof}

The proof of the previous theorem gives us a interesting result about the \textit{Centroid} of a Lie algebra, usually denoted by $\mathcal{C}(\mathfrak{g})$, and which is the space $\mathcal{D}(1,1,0)(\mathfrak{g})$.

\begin{theorem}\label{centroid}
Let $t\in \mathbb{K}\setminus\{0,1\}$ and let $\mathfrak{g}$ be a non-perfect Lie algebra over $\mathbb{K}$. Then $\mathcal{D}(t,1,0)(\mathfrak{g})$  can be embedded into the centroid of $\mathfrak{g}$.
\end{theorem}

The preceding results yield our second objective.

\begin{theorem}
   Let $t,s\in \mathbb{K}\setminus\{0,1\}$ be arbitrary. Then the functions $\widehat{\phi}_{n,t}$ and $\widehat{\phi}_{n,s}$ are equal.
\end{theorem}

\begin{remark}
Because of this theorem, the family $\{\widehat{\phi}_{n,t}\}_{t \in \mathbb{K}}$ ``collapses'' to $\{\widehat{\phi}_{n,-1},\widehat{\phi}_{n,0},\widehat{\phi}_{n,1}\}$, and it is easy to check that these three functions are different from each other by evaluating the functions at a Lie algebra law. Consider the Lie algebra $\mathfrak{g}:=\mathfrak{aff}(\mathbb{K})\times \mathbb{K}^{n}$ (with $n\in \mathbb{N})$). It is clear that $\mathcal{D}(0,1,0)(\mathfrak{g})$ is isomorphic to $L(\mathbb{K}^2\times\mathbb{K}^{n},\mathbb{K}^{n})$, and by Lemma \ref{lematecnico} we have $\mathcal{D}(1,1,0)(\mathfrak{g})$ is isomorphic to $\mathcal{D}(1,1,0)(\mathfrak{aff}(\mathbb{K}))\oplus L(\mathbb{K}\times \mathbb{K}^{n},\mathbb{K}^{n})$ and $\mathcal{D}(-1,1,0)(\mathfrak{g})$ is isomorphic to $\mathcal{D}(-1,1,0)(\mathfrak{aff}(\mathbb{K}))\oplus L(\mathbb{K}\times \mathbb{K}^{n},\mathbb{K}^{n})$. Since $\mathcal{D}(1,1,0)(\mathfrak{aff}(\mathbb{K})) = \operatorname{span}\{\operatorname{Id}_{\mathfrak{aff}(\mathbb{K})}\}$ and $\mathcal{D}(-1,1,0)(\mathfrak{aff}(\mathbb{K})) = \{0\}$, where $\operatorname{Id}_{\mathfrak{aff}(\mathbb{K})}$ denotes the \textit{identity map} of $\mathfrak{aff}(\mathbb{K})$, then the statement follows.
\end{remark}

\section{Upper and lower bounds of $\widehat{\phi}_{n,t}$, $t\neq 0,1$}

\begin{theorem}\label{bounds}
Let $\mathfrak{g}=(V,\mu)$ be a non-perfect Lie algebra and let $t=-1$ (or any $t\neq 0,1$). Then
$$
\operatorname{dim}(\Omega(\mathfrak{g})) \leq \widehat{\phi}_{n,t}(\mathfrak{g}) \leq \operatorname{dim}(\Omega(\mathfrak{g})) + \operatorname{dim}(L (\mathfrak{g}^{(2)}; \mathcal{Z}(\mathfrak{g})\cap  \mathfrak{g}^{(2)} ) ).
$$

\end{theorem}

\begin{proof}
  The idea of the proof is similar to that of the Lemma \ref{lematecnico}. Let $D$ be a $(t,1,0)$-derivation of $\mathfrak{g}$ and consider the linear transformation $\widehat{D}:\mathfrak{g}^{(2)} \rightarrow \mathcal{Z}(\mathfrak{g})\cap  \mathfrak{g}^{(2)}$ defined by $\widehat{D} X = D X$ (formally and rigorously, $\widehat{D}$ is not $D \circ \iota_{\mathfrak{g}^{(2)}}$). The function $\widehat{D}$ is well defined because $D$ preserves $\mathfrak{g}^{(2)}$, and by Lemma \ref{lematecnico2}, $D$ sends $\mathfrak{g}^{(2)}$ to $\mathcal{Z}(\mathfrak{g})$. Now, consider the function $\varphi: \mathcal{D}(t,1,0)(\mathfrak{g}) \rightarrow  L(\mathfrak{g}^{(2)}; \mathcal{Z}(\mathfrak{g})\cap  \mathfrak{g}^{(2)} )$ that sends $D$ to $\widehat{D}$. The function $\varphi$ is a linear transformation that passes to the quotient $ \mathcal{D}(t,1,0)(\mathfrak{g}) /\Omega(\mathfrak{g})$ and so, it defines a linear transformation $\widetilde{\varphi}: \mathcal{D}(t,1,0)(\mathfrak{g}) /\Omega(\mathfrak{g}) \rightarrow L(\mathfrak{g}^{(2)}; \mathcal{Z}(\mathfrak{g})\cap  \mathfrak{g}^{(2)} )$.

  The only thing that remains to be proved is that $\widetilde{\varphi}$ is a injective function. Suppose $D \in \mathcal{D}(t,1,0)(\mathfrak{g})$ such that $\widehat{D}$ is the zero map from $\mathfrak{g}^{(2)}$ to  $\mathcal{Z}(\mathfrak{g})\cap  \mathfrak{g}^{(2)}$. We will show that $D \in \Omega(\mathfrak{g})$. Clearly $D$ satisfies the condition $\mathfrak{g}^{(2)} \subseteq \operatorname{Ker}(D)$.  Now we show that $\operatorname{Im}(D)\subseteq \mathcal{Z}(\mathfrak{g})$. Let $X, Y$  be arbitrary. Since $D$ is $(t,1,0)$-derivation of $\mathfrak{g}$, we have $\mu(DX,Y)=tD\mu(X,Y)$, and since $\mu(X,Y) \in \mathfrak{g}^{2}$, it follows that $\mu(DX,Y)=t\widehat{D}\mu(X,Y)=0$; and thus $D X \in \mathcal{Z}(\mathfrak{g})$.
\end{proof}

\begin{remark}
The preceding theorem is sharp in all dimensions. For instance, take $\mathfrak{g}$ to be $\mathfrak{h}_{3}(\mathbb{K})\times \mathbb{K}^{n}$, with $n\in \mathbb{N} \cup \{0\}$. If $\{e_1,e_2,e_3\}$ is a basis of $\mathfrak{h}_{3}(\mathbb{K})$ and $\mu(e_1,e_2)=e_3$, we consider the linear transformation $D \in L(\mathfrak{g};\mathfrak{g})$ defined by $e_1 \mapsto t e_1$, $e_2 \mapsto t e_2$, $e_3 \mapsto e_3$ and $\mathbb{K}^{n}\mapsto 0$ (when $n\neq0$). It is clear that $D$ is a $(t,1,0)$-derivation of $\mathfrak{g}$ which is not in $\Omega(\mathfrak{g})$. Thus, $1 + \operatorname{dim}(\Omega(\mathfrak{g})) \leq \widehat{\phi}_{3+n,t}(\mathfrak{g})$, and since $\mathfrak{g}^{(2)}=\mathcal{Z}(\mathfrak{g})\cap  \mathfrak{g}^{(2)}$ is a $1$-dimensional vector space, it follows from Theorem \ref{bounds} that the inequality becomes equality; i.e. $\widehat{\phi}_{3+n,t}(\mathfrak{g})$ is equal to $\operatorname{dim}(\Omega(\mathfrak{g})) + \operatorname{dim}(L (\mathfrak{g}^{(2)}; \mathcal{Z}(\mathfrak{g})\cap  \mathfrak{g}^{(2)} ) )$.

On the other hand, as we mentioned above, if $\mathfrak{h}=\mathfrak{aff}(\mathbb{K})\times \mathbb{K}^{m}$, with $m\in \mathbb{N}$, then $\widehat{\phi}_{2+m,t}(\mathfrak{g}) = \operatorname{dim}(\Omega(\mathfrak{g}))$.
\end{remark}

\begin{remark}
Theorem \ref{bounds} can be improved under additional assumptions about $\mathfrak{g}^{(2)}$. If $\mathfrak{g}_{2}:=\mu(\mathfrak{g}^{(2)},\mathfrak{g})$ is a nontrivial proper subset of $\mathfrak{g}^{(2)}$ (for instance, if $\mathfrak{g}$ is a $k$-step nilpotent Lie algebra, with $k\geq 3$), then we can change $\operatorname{dim}(L (\mathfrak{g}^{(2)}; \mathcal{Z}(\mathfrak{g})\cap  \mathfrak{g}^{(2)} ) )$ to $\operatorname{dim}(L (\mathfrak{g}^{(2)}/\mathfrak{g}_{2}; \mathcal{Z}(\mathfrak{g})\cap  \mathfrak{g}^{(2)} ) )$ in the conclusion of the theorem, because of Lemma \ref{lematecnico2}. 
\end{remark}

\begin{remark}
Under the assumption that $\mathfrak{g}_{2}:=\mu(\mathfrak{g}^{(2)},\mathfrak{g})$ is a nontrivial proper subset of $\mathfrak{g}^{(2)}$, since the identity map of $\mathfrak{g}$ is not in the image of the embedding of $\mathcal{D}(t,1,0)(\mathfrak{g})$ into $\mathcal{D}(1,1,0)(\mathfrak{g})=\mathcal{C}(\mathfrak{g})$  given in the proof of Theorem \ref{centroid}, we also have the upper bound $\widehat{\phi}_{n,t}(\mathfrak{g}) + 1 \leq \widehat{\phi}_{n,1}(\mathfrak{g})$, with $t\neq 0,1$. This bound is sharp in all dimensions: take $\mathfrak{g}$ to be the \textit{{standard filiform Lie algebra}} of dimension $n$ ($n\geq 4$), which is the semidirect product $\mathbb{K} \mathlarger{\mathlarger{\oset{\mbox{\footnotesize T}}{\ltimes}}} \mathbb{K}^{n}$ where $T \in L(\mathbb{K}^n;\mathbb{K}^n)$ is any nilpotent linear transformation with \textit{nilpotency index} $n$. An straightforward computation shows that $\mathcal{D}(1,1,0)(\mathfrak{g})= \operatorname{span}\{\operatorname{Id}_{\mathfrak{g}}\} \oplus \Omega(\mathfrak{g})$ and $\mathcal{D}(1,t,0)(\mathfrak{g})= \Omega(\mathfrak{g})$.
\end{remark}

\begin{remark}
Let $D$ be a $(t,1,0)$-derivation of a Lie algebra $\mathfrak{g}=(V,\mu)$ and let $\lambda \in C^{2}(V;V)$ be the skew-symmetric bilinear map defined by $\lambda(X,Y)=\mu(DX,Y)$.  It is a simple matter to see that $\lambda$ is a \textit{trivial solution} to the \textit{Maurer-Cartan (deformation) equation} of $\mathfrak{g}$
$$
\delta_{\mu}+\frac{1}{2}[\lambda,\lambda]=0;
$$
which is to say that $\lambda$ is a $2$-cocycle for the adjoint representation of the Lie algebra $\mathfrak{g}$ and $(V,\lambda)$ is a Lie algebra; and so $\mu + s \lambda \in \mathcal{L}(V)$ for all $s\in \mathbb{K}$. 

Thus, the results obtained provides information about this kind of \textit{linear deformations}.
\end{remark}



\end{document}